\documentclass[12pt]{amsart}

\usepackage{psfrag}
\usepackage{graphicx}
\usepackage{pinlabel}
\usepackage{amsmath}
\usepackage{amssymb}
\usepackage{amscd}
\usepackage{picinpar}
\usepackage{times}
\usepackage{pb-diagram}
\usepackage{graphicx}
\usepackage{wrapfig}
\usepackage{xspace}
\usepackage{hyperref}
\usepackage{xcolor}

\newtheorem{theorem}{Theorem}[section]
\newtheorem{lemma}[theorem]{Lemma}

\newtheorem{definition}[theorem]{Definition}
\newtheorem{conj}[theorem]{Conjecture}

\begin{document}

\title{The deformation Space of Geodesic Triangulations and Generalized Tutte's Embedding Theorem}

\author{Yanwen Luo, Tianqi Wu, Xiaoping Zhu}
\thanks{Acknowledgement: The author was supported in part by NSF 1760471, NSF DMS FRG 1760527 and NSF DMS 1811878.}


\begin{abstract}
We proved the contractibility of the deformation space of the geodesic triangulations on a closed surface of negative curvature. This solves an open problem proposed by Connelly et al. in 1983 \cite{connelly1983problems}, in the case of hyperbolic surfaces. The main part of the proof is a generalization of Tutte's embedding theorem for closed surfaces of negative curvature.

\end{abstract}

\maketitle

\section{Introduction}

In this paper, we study the deformation space of geodesic triangulations of a surface within a fixed homotopy type. Such space can be viewed as a discrete analogue of the space of surface diffeomorphisms homotopic to the identity. Our main theorem is the following.
\begin{theorem}
For a closed orientable surface of negative curvature, the space of geodesic triangulations in a homotopy class is contractible. In particular, it is connected.
\end{theorem}

The group of diffeomorphisms of a smooth surface is one of the fundamental objects in the study of low dimensional topology. Determining the homotopy types of diffeomorphism groups has profound implications to a wide range of problems in  Teichmuller spaces, mapping class groups, and  geometry and topology of 3-manifolds.
Smale \cite{smale1959diffeomorphisms} proved that the group of diffeomorphisms of a closed 2-disk which fix the boundary pointwisely is contractible. This enables him to show that the group of orientation-preserving diffeomorphisms of the 2-sphere is homotopic equivalent to $SO(3)$ \cite{smale1959diffeomorphisms}.
Earle-Eells \cite{earle1969fibre} identified the homotopy type of the group of the diffeomorphisms homotopic to the identity for any closed surface. In particular, such topological group is contractible for a closed orientable surface with genus greater than one, consisting with our Theorem 1.1 for the discrete analogue.

Cairns \cite{cairns1944isotopic} initiated the investigation of the topology of the space of geodesic triangulations, and proved that if the surface is a geometric triangle in the Euclidean plane, the space of geodesic triangulations with fixed boundary edges is connected. A series of further developments culminated in a discrete version of Smale's theorem proved by Bloch-Connelly-Henderson \cite{bloch1984space} as follows.

\begin{theorem}
The space of geodesic triangulations of a convex polygon with fixed boundary edges is homeomorphic to some Euclidean space. In particular, it is contractible.
\end{theorem}
A simple proof of the contractibility of the space above is provided in \cite{1910.03070} using Tutte's embedding theorem \cite{tutte1963draw}. It also provides examples showing that the homotopy type of this space can be complicated if the boundary of the polygon is not convex.
For closed surfaces, it is conjectured in \cite{connelly1983problems} that
\begin{conj}
The space of geodesic triangulations of a closed orientable surface  with constant curvature deformation retracts to the group of isometries of the surface homotopic to the identity.
\end{conj}
The connectivity of these spaces has been explored in \cite{cairns1944isotopic, chambers2021morph, hass2012simplicial}. Awartani-Henderson \cite{awartani1987spaces} identified a contractible subspace in the space of geodesic triangulations of the 2-sphere. Hass-Scott \cite{hass2012simplicial} showed that the space of geodesic triangulation of a surface with a hyperbolic metric is contractible if the triangulation contains only one vertex.
The main result of this paper affirms conjecture 1.3 in the case of hyperbolic surfaces.

\subsection{Set Up and the Main Theorem}
Assume $M$ is a connected closed orientable smooth surface with a smooth Riemannian metric $g$ of non-positive Gaussian curvature.  A topological triangulation of $M$ can be identified as a  homeomorphism $\psi$ from $|T|$ to $M$, where $|T|$ is the carrier of a 2-dimensional simplicial complex $T=(V,E,F)$ with the vertex set $V$, the edge set $E$, and the face set $F$. 
For convenience, we label the vertices as $1,2,...,n$ where $n=|V|$ is the number of vertices.
The edge in $E$ determined by vertices $i$ and $j$ is denoted as $ij$. Each edge is identified with the closed unit interval $[0,1]$.

Let $T^{(1)}$ be the 1-skeleton of $T$, and denote $X=X(M,T,\psi)$ as the space of geodesic triangulations homotopic to $\psi|_{T^{(1)}}$. More specifically, $X$ contains all the embeddings $\varphi:T^{(1)}\rightarrow M$ satisfying that
\begin{enumerate}
	\item The restriction $\varphi_{ij}$ of $\varphi$ on the edge $ij$ is a geodesic parameterized with constant speed, and
	\item $\varphi$ is homotopic to $\psi|_{T^{(1)}}$.
\end{enumerate}

It has been proved by Colin de Verdi{\`e}re \cite{de1991comment} that such $X(M, T,\psi)$ is always non-empty. Further, $X$ is naturally a metric space, with the distance function
$$
d_X(\varphi,\phi)=\max_{x}d_{g}(\varphi(x),\phi(x)).
$$
Then our main theorem is formally stated as follows.
\begin{theorem}
\label{main}
If $(M,g)$ has strictly negative Gaussian curvature, then $X(M,T,\psi)$ is contractible. In particular, it is connected.
\end{theorem}

\subsection{Generalized Tutte's Embedding}

Let $\tilde X=\tilde X (M,T,\psi)$ be the super space of $X$, containing all the continuous maps $\varphi:T^{(1)}\rightarrow M$ satisfying that
\begin{enumerate}
	\item The restriction $\varphi_{ij}$ of $\varphi$ on the edge $ij$ is geodesic parameterized with constant speed, and
	\item $\varphi$ is homotopic to $\psi|_{T^{(1)}}$.
\end{enumerate}
Notice that elements in $\tilde X$ may not be embeddings of $T^{(1)}$ to $M$. The space $\tilde{X}$ is also naturally a metric space, with the same distance function
$$
d_{\tilde X}(\varphi,\phi)=\max_{x}d_{g}(\varphi(x),\phi(x)).
$$
We call an element in $\tilde{X}$ a \textit{geodesic mapping}. A geodesic mapping is determined by the positions $q_i = \varphi(i)$ of the vertices and the homotopy classes of $\varphi_{ij}$ relative to the endpoints $q_i$ and $q_j$. In particular, this holds for geodesic triangulations. Since we can perturb the vertices of a geodesic triangulation to generate another, $X$ is a $2n$ dimensional manifold.

Let $(i, j)$ be the directed edge starting from the vertex $i$ and ending at the vertex $j$. Denote
$
\vec E=\{(i,j):ij\in E\}
$
as the set of directed edges of $T$. A positive vector $w\in\mathbb R^{\vec E}_{>0}$ is called a \emph{weight} of $T$. For any weight $w$ and geodesic mapping $\varphi\in \tilde X$, we say $\varphi$ is \textit{$w$-balanced} if for any $i\in V$,
$$
\sum_{j:ij\in E}w_{ij}{v}_{ij}=0.
$$
Here $ v_{ij}\in T_{q_i}M$ is defined with the exponential map $\exp:TM\to M$ such that $\exp_{q_i}(t{v}_{ij}) = \varphi_{ij}(t)$ for $t\in[0, 1]$.

The main part of the proof of Theorem \ref{main} is to generalize  Tutte's embedding theorem (see Theorem 9.2 in \cite{tutte1963draw} or Theorem 6.1 in \cite{floater2003one}) to closed surfaces of negative curvature. Specifically, we will prove
the following two theorems.
\begin{theorem}
\label{existence}
Assume $(M,g)$ has strictly negative Gaussian curvature. For any weight $w$, there exists a unique geodesic mapping $\varphi\in \tilde X(M, T,\psi)$ that is $w$-balanced. Such induced map $\Phi(w)=\varphi$ is continuous from $\mathbb R^{\vec E}_{>0}$ to $\tilde X$.
\end{theorem}

\begin{theorem}
\label{embedding}
If $\varphi\in\tilde X$ is $w$-balanced for some weight $w$, then $\varphi\in X$.
\end{theorem}

Theorem \ref{embedding} can be regarded as a generalization of the embedding theorems by Colin de Verdi{\`e}re (see Theorem 2 in \cite{de1991comment}) and Hass-Scott (see Lemma 10.12 in \cite{hass2012simplicial}), which imply that the minimizer of the following discrete Dirichlet energy
$$E(\varphi) = \frac{1}{2}\sum_{ij\in E}w_{ij}l^2_{ij}$$
among the maps $\varphi$ in the homotopy class of $\psi|_{T^{(1)}}$ is a geodesic triangulation. Here $l_{ij}$ is the geodesic length of $\varphi_{ij}$ in $M$. The minimizer is a $w$-balanced geodesic mapping with $w_{ij} = w_{ji}$ for $ij\in E$. Hence, Theorem \ref{embedding} extends the previous results from the cases of symmetric weights to non-symmetric weights.
We believe that, the proofs in Colin de Verdi{\`e}re \cite{de1991comment} and Hass-Scott \cite{hass2012simplicial} could be easily modified to work with our non-symmetric case. Nevertheless, we will give a new proof in Section 3 to make the paper self-contained.

\subsection{Mean Value Coordinates and the Proof of Theorem \ref{main}}

Theorem \ref{existence} and \ref{embedding} give a continuous map $\Phi$ from $\mathbb R^{\vec E}_{>0}$ to $X$.
For the oppositie direction, we can construct a weight $w$ for a geodesic embedding $\varphi\in X$, using \emph{mean value coordinates} which was firstly introduced by Floater \cite{floater2003mean}. Given $\varphi\in X$, the mean value coordinates are defined to be
$$
w_{ij}=\frac{\tan(\alpha_{ij}/2)+\tan(\beta_{ij}/2)}{| v_{ij}|},
$$
where $| v_{ij}|$ equals to the geodesic length of $\varphi_{ij}([0,1])$, and $\alpha_{ij}$ and $\beta_{ij}$ are the two inner angles in $\varphi(T^{(1)})$ at the vertex $\varphi(i)$ sharing the edge $\varphi_{ij}([0,1])$. The construction of mean value coordinates gives a continuous map $\Psi$ from $X$ to $\mathbb R^{\vec E}_{>0}$. Further,
by Floater's mean value theorem (see Proposition 1 in \cite{floater2003mean}),
any $\varphi\in X$ is $\Psi(\varphi)$-balanced.
Namely, $\Phi\circ \Psi=id_X$. Then Theorem \ref{main} is a direct consequence of Theorem \ref{existence} and \ref{embedding}.

\begin{proof}[Proof of Theorem \ref{main}]
Since $\mathbb R^{\vec E}_{>0}$ is contractible, $\Psi\circ\Phi$ is homotopic to the identity map. Since $\Phi\circ \Psi=id_X$, $X$ is homotopy equivalent to the contractible space $\mathbb R^{\vec E}_{>0}$.
\end{proof}
In the remaining of the paper, we will prove Theorem \ref{existence} in Section 2 and Theorem \ref{embedding} in section 3.

\section{Proof of Theorem \ref{existence}}
Theorem \ref{existence} consists of three parts: the existence of $w$-balanced geodesic mapping, the uniqueness of $w$-balanced geodesic mapping and the continuity of the map $\Phi$.
In this section, we will first parametrize $\tilde X$ by $\tilde M^n$, where $\tilde M$ is the universal covering of $M$, and then prove the three parts in Subsection 2.1 and 2.2 and 2.3 respectively.

Assume that $p$ is the covering map from $\tilde M$ to $M$, and $\Gamma$ is the corresponding group of deck transformations of the covering so that $\tilde{M}/\Gamma = M$. For any $i\in V$, fix a lifting $\tilde q_i\in\tilde M$ of $q_i\in M$. For any edge $ij$, denote $\tilde \varphi_{ij}(t)$ as the lifting of $\varphi_{ij}(t)$ such that $\tilde \varphi_{ij}(0) =  \tilde q_i$. Then $p(\tilde \varphi_{ij}(1)) = \varphi_{ij}(1)=q_j=p(\tilde q_j)$, and there exists a unique deck transformation $A_{ij}\in\Gamma$ such that $\tilde \varphi_{ij}(1) = A_{ij} \tilde q_j$. It is easy to see that $A_{ij}=A_{ji}^{-1}$ for any edge $ij$.

Equip $\tilde M$ with the natural pullback Riemannian metric $\tilde g$ of $g$ with negative Gaussian curvature. This metric is equivariant with respect to $\Gamma$.  For any $ x, y\in\tilde M$, there exists a unique geodesic with constant speed parameterization $\gamma_{x,y}:[0,1]\rightarrow\tilde M$ such that $\gamma_{x,y}(0)=x$ and $\gamma_{x,y}(1)=y$. We can naturally parametrize $\tilde X$ as follows.

\begin{theorem}
For any $( x_1,..., x_n)\in\tilde M^n$, define $\varphi=\varphi[x_1,...,x_n]$ as
$$
\varphi_{ij}(t)= p\circ\gamma_{x_i,A_{ij}x_j}(t)
$$
for any $ij\in E$ and $t\in[0,1]$.
Then such $\varphi$ is a well-defined geodesic mapping in $\tilde{X}$, and the map $(x_1,...,x_n)\mapsto\varphi[x_1,...,x_n]$ is a homeomorphism from  $\tilde M^n$ to $\tilde X$.
\end{theorem}
Here we omit the proof of Theorem 2.1 which is routine but lengthy.
In the remaining of this section,
for any $x,y,z\in\tilde M$ and $u,v\in T_x\tilde M$, we denote
\begin{enumerate}

\item  $d(x,y)$ as the intrinsic distance between $x,y$ in $(\tilde M,\tilde g)$, and

\item  $v(x,y)=\exp_x^{-1}y\in T_x\tilde M$, and

\item  $\triangle xyz$ as the geodesic triangle in $\tilde M$ with vertices $x,y,z$, which could possibly be degenerate, and

\item  $\angle yxz$ as the inner angle of $\triangle xyz$ at $x$ if $d(x,y)>0$ and $d(x,z)>0$, and
\item  $|v|$ as the norm of $v$ under the metric $\tilde g_x$, and

\item  $u\cdot v$ as the inner product of $u$ and $v$ under the metric $\tilde g_x$.

\end{enumerate}
By scaling the metric if necessary, we may assume that the Gaussian curvatures of $(M,g)$ and $(\tilde M,\tilde g)$ are bounded above by $-1$.

\subsection{Proof of the Uniqueness}
We first prove the following Lemma \ref{CAT} using CAT($0$) geometry. See Theorem 4.3.5 in  \cite{burago2001course} and Theorem 1A.6 in \cite{bridson2013metric} for the well-known comparison theorems.

\begin{lemma}
\label{CAT}
Assume $x,y,z\in\tilde M$, then
\begin{enumerate}
	\item $|v(z,x) -  v(z,y)|\leq d(x,y)$, and
	\item $ v(x,y)\cdot  v(x,z)+ v(y,x)\cdot v(y,z)\geq d(x,y)^2$,
\end{enumerate}
and the equality holds if and only if $\triangle xyz$ is degenerate.
\end{lemma}

\begin{proof}
If $\triangle xyz$ is degenerate, then there exists a geodesic $\gamma$ in $\tilde M$ such that $x,y,z\in\gamma$, and then the proof is straightforward. So we assume that $\triangle xyz$ is non-degenerate.

(1) Three points $v(z,x), v(z,y)$, and $0$ in $T_z\tilde M$ determine a Euclidean triangle, where $|v(z,x)|=d(x,z)$, and $|v(z,y)|=d(z,y)$ and the angle between $v(z,x)$ and $v(z,y)$ is equal to $\angle xzy$. Then by the CAT(0) comparison theorem,
$$
|v(z,x)-v(z,y)|< d(x,y).
$$

(2) Let $x',y',z'\in\mathbb R^2$ be such that
$$
|x'-z'|_2=|v(x,z)|,\quad\quad\quad |y'-z'|_2=|v(y,z)|,\quad\text{and}\quad|x'-y'|_2=|v(x,y)|.
$$
Then by the CAT$(0)$ comparison theorem,
$\angle yxz< \angle y'x'z'$, and $\angle xyz< \angle x'y'z'$.
Hence,
$$
v(x,y)\cdot  v(x,z)+ v(y,x)\cdot v(y,z)>
(y'-x')\cdot(z'-x')+(x'-y')\cdot(z'-y')=|x'-y'|_2^2
= d(x,y)^2.
$$
\end{proof}

\begin{proof}[Proof of the uniqueness part in Theorem \ref{existence} ]
If not, assume $\varphi[x_1,...,x_n]$ and $\varphi[x_1',...,x_n']$ are two different geodesic mappings that are both $w$-balanced for some weight $w$.
We are going to prove a discrete maximum principle for the function $j\mapsto d(x_j,x_j')$.
Assume $i\in V$ is such that $d(x_i,x_i')=\max_{j\in V}d(x_j,x_j')>0$. By lifting the $w$-balanced assumption to $\tilde M$, we have that

\begin{equation}
\sum_{j:ij\in E}w_{ij} v(x_i,A_{ij}x_j)=0,
\end{equation}
and
\begin{equation}
\sum_{j:ij\in E}w_{ij}v(x_i',A_{ij}x_j')=0.
\end{equation}

Then by part (1) of Lemma \ref{CAT} and equation (1),
\begin{align*}
&\left|\sum_{j:ij\in E}w_{ij}v(x_i,A_{ij}x_j')\right| \\
= &\left|\sum_{j:ij\in E}w_{ij}v(x_i,A_{ij}x_j') - \sum_{j:ij\in E}w_{ij}v(x_i,A_{ij}x_j)\right|\\
\leq&\sum_{j:ij\in E}w_{ij} d(A_{ij}x_j,A_{ij}x_j')\\
=&\sum_{j:ij\in E}w_{ij} d(x_j,x_j')\\
\leq &d(x_i,x_i')\sum_{j:ij\in E}w_{ij}.
\end{align*}

By part (2) of Lemma \ref{CAT}, equation (2), and the Cauchy-Schwartz inequality,
\begin{align*}
&d(x_i,x_i')\cdot\left|\sum_{j:ij\in E}w_{ij}v(x_i,A_{ij}x_j')\right| \\
\geq &v(x_i,x_i')\cdot\sum_{j:ij\in E}w_{ij}v(x_i,A_{ij}x_j') +
v(x_i',x_i)\cdot\sum_{j:ij\in E}w_{ij}v(x_i',A_{ij}x_j')\\
\geq  &\sum_{j:ij\in E}w_{ij}\cdot d(x_i,x_i')^2.
\end{align*}

Therefore, the equalities hold in both inequalities above. Then
for any neighbor $j$ of $i$, $d(x_j,x_j')=d(x_i,x_i')=\max_{k\in V}d(x_k,x_k')$, and
$A_{ij}x_j$ is on the geodesic determined by $x_i$ and $x_i'$. Then the one-ring neighborhood of $p(x_i)$ in $\varphi[x_1,...,x_n](T^{(1)})$ degenerates to a geodesic arc.
By the connectedness of the surface, we can repeat the above argument and deduce that $d(x_j,x_j')=d(x_i,x_i')$ for any $j\in V$. Further, for any triangle $\sigma\in F$, $\varphi[x_1,...,x_n](\partial\sigma)$ degenerates to a geodesic arc.

It is not difficult to extend $\varphi[x_1,...,x_n]$ to a continuous map $\tilde\varphi$ from $|T|$ to $M$, such that for any triangle $\sigma\in F$, $\tilde\varphi(\sigma)=\varphi[x_1,...,x_n](\partial\sigma)$ being a geodesic arc.

It is also not difficult to prove that $\tilde\varphi$ is homotopic to $\psi$. Therefore, $\tilde\varphi$ is
degree-one  and surjective. This is contradictory to that $\tilde\varphi(|T|)$ is a finite union of geodesic arcs.
\end{proof}

\subsection{Proof of the Existence}
Here we prove a stronger existence result.

\begin{theorem}
\label{proper}
Given a compact subset $K$ of $\mathbb R^{\vec E}_{>0}$, there exists a compact subset $K'=K'(M,T,\psi,K)$ of $\tilde X$ such that for any $w\in K$, there exists a $w$-balanced geodesic mapping $\varphi\in K'$.
\end{theorem}

We first introduce a topological Lemma \ref{toplemma} and then reduce Theorem \ref{proper} to Lemma \ref{keylemma}.

\begin{lemma}
\label{toplemma}
Suppose $B^{n}=\{x\in\mathbb R^n: |x|\leq1\}$ is the unit ball in $\mathbb R^n$, and $f:B^{n}\rightarrow\mathbb R^n$ is a continuous map such that $x\neq f(x)/|f(x)|$ for any $x\in\partial B^n=S^{n-1}$ with $f(x)\neq0$. Then $f$ has a zero in $B^n$.
\end{lemma}

\begin{proof}
If not, $g(x) = f(x)/|f(x)|$ is a continuous map from $B^n$ to $\partial B^n$. Since $B^n$ is contractible, $g(x)$ is null-homotopic, and thus $g|_{S^{n-1}}$ is also null-homotopic. Since $g(x)\neq x$,
it is easy to verify that
$$
H(x,t)=\frac{tg(x)+(1-t)(-x)}{|tg(x)+(1-t)(-x)|}
$$
is a homotopy between $g|_{S^{n-1}}$ and $-id|_{S^{n-1}}$. This contradicts to that $-id|_{S_{n-1}}$ is not null-homotopic.
\end{proof}

\begin{lemma}
\label{keylemma}
We fix an arbitrary point $q\in\tilde M$. If $w\in \mathbb R^{\vec E}_{>0}$ and $(x_1,...,x_n)\in\tilde M^n$ satisfies that
\begin{equation}
\label{residue}
 v (x_i,q)\cdot\sum_{j:ij\in E}w_{ij}  v(x_i,A_{ij}x_j)\leq 0
\end{equation}
for any $i\in V$, then
$$
\sum_{i\in V}d(x_i,q)^2< R^2
$$
for some constant $R>0$ which depends only on $M,T,\psi,q$ and
$$
\lambda_w:=\frac{\max_{ij\in E} w_{ij}}{\min_{ij\in E} w_{ij}} .
$$
\end{lemma}

The vector in Figure \ref{residuev}
$$ r_i = \sum_{j:ij\in E}w_{ij}  v(x_i,A_{ij}x_j)$$
is defined as the \textit{residue vector} $ r_i$ at $x_i$ of $\varphi[x_1, \cdots, x_n]$ with respect to the weight $w$.  Notice that a geodesic mapping $\varphi$ is $w$-balanced if and only if all its residue vectors vanish with respect to $w$. Lemma \ref{keylemma} means that if all the residue vectors are dragging $x_i$'s away from $q$, then all the $x_i$'s must stay not far away from $q$.

The definition of residue vector is similar to the concept of \textit{discrete tension field} in \cite{gaster2018computing}.

\begin{figure}[h!]
  \includegraphics[width=0.4\linewidth]{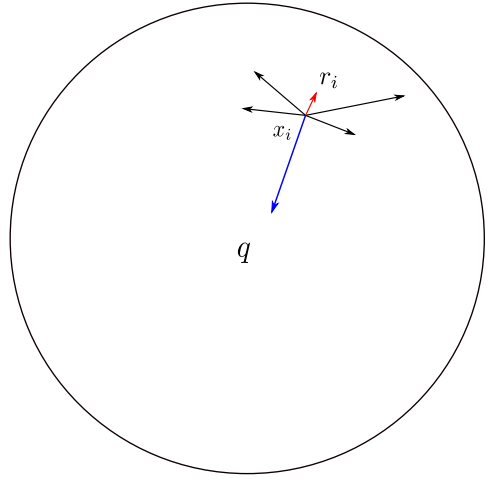}
  \caption{The residue vector and Lemma \ref{keylemma}.}
    \label{residuev}
\end{figure}

\begin{proof}[Proof of Theorem \ref{proper}]
Fix an arbitrary base point $q\in\tilde M$, and then by Lemma \ref{keylemma} we can pick a sufficiently large constant $R=R(M,T,\psi,K)>0$ such that if
$$
\sum_{i=1}^n d(x_i,q)^2=R^2,
$$
there exists $i\in V$ such that
$$
 v (x_i,q)\cdot\sum_{j:ij\in E}w_{ij}  v(x_i,A_{ij}x_j)> 0.
$$
We will prove that the compact set
$$
K'=\{\varphi[x_1,...,x_n]:\sum_{i=1}^n d(x_i,q)^2\leq R^2\}
$$
is satisfactory.

For any $x\in \tilde M$, let $P_x:T_x\tilde M\rightarrow T_q\tilde M$ be the parallel transport along the geodesic $\gamma_{x,q}$. Set
$$
B=\{( v_1,..., v_n)\in (T_q\tilde M)^n:\sum_{i=1}^n| v_i|^2\leq 1\}
$$
as a Euclidean $2n$-dimensional unit ball,
and construct a map
$
F:B\rightarrow (T_q\tilde M)^n
$
in the following three steps. Firstly, we construct $n$ points $x_1,...,x_n\in\tilde M$ as
$
x_i(v_1,...,v_n)=\exp_q(Rv_i).
$
Secondly, we compute the residue vector at each $x_i$ as
$$
r_i=\sum_{j:ij\in E} w_{ij}  v(x_i,A_{ij}x_j)\in T_{x_i}\tilde M.
$$
Lastly, we pull back the residues to $T_q\tilde M$ as
$F(v_1,...,v_n) = \left(P_{x_1}(r_1),...,P_{x_n}(r_n)\right).$

Notice that the map $(v_1,...,v_n)\mapsto \varphi[x_1,...,x_n]$ is a homeomorphism from $B$ to $K'$, and $F(v_1,...,v_n)=0$ if and only if the corresponding $\varphi[x_1,...,x_n]$ in $K'$ is $w$-balanced map. Hence, it suffices to prove that $F$ has a zero in $B$. By Lemma \ref{toplemma} it suffices to prove that for any $(v_1,...,v_n)\in\partial B$,
$$
(v_1,...,v_n)\neq \frac{F(v_1,..,v_n)}{|F(v_1,...,v_n)|}.
$$
Suppose $(v_1,...,v_n)$ is an arbitrary point on $\partial B$, and then
it suffices to prove that there exists $i\in V$ such that
$v_i\cdot F_i(v_1,...,v_n)=v_i\cdot P_{x_i}(r_i)<0$.

Notice that $x_1(v_1,...,v_n),...,x_n(v_1,...,v_n)$ satisfy that $\sum_{i=1}^nd(q,x_i)^2=R^2$, so by our assumption on $R$, there exists $i\in V$ such that
$$
v (x_i,q)\cdot\sum_{j:ij\in E}w_{ij}  v(x_i,A_{ij}x_j)=v(x_i,q)\cdot r_i> 0,
$$
and thus,
$$
v_i\cdot P_{x_i}(r_i)=-\frac{1}{d(q,x_i)}P_{x_i}\left( v(x_i,q)\right)\cdot P_{x_i}(r_i)
=-\frac{1}{d(q,x_i)} v(x_i,q)\cdot r_i< 0.
$$
\end{proof}

In the rest of this subsection, we will prove Lemma \ref{keylemma} by contradiction. Let us first sketch the idea of the proof. Assume
$\sum_{i\in V}d(x_i,q)^2$ is very large,
then by a standard compactness argument, there exists a long edge $ij$ in the geodesic mapping $\varphi[x_1,..., x_n]$. Assume $d(q,x_i)\geq d(q,x_j)$, then the corresponding long edge $\gamma_{x_i,A_{ij}x_j}$ in $\tilde M$ is pulling $x_i$ towards $q$. This implies that  there exists another long edge $\gamma_{x_i,A_{ik}x_k}$ dragging $x_i$ away from $q$, otherwise the residue vector $r_i$ would not drag $x_i$ away from $q$. It can be shown that $d(q,x_k)>d(q,x_i)$.
Repeating the above steps, we can find an arbitrary long sequence of vertices such that the distance from each of these vertices to $q$ is increasing. This is impossible as we only have finitely many vertices.

\begin{figure}[h!]
  \includegraphics[width=0.8\linewidth]{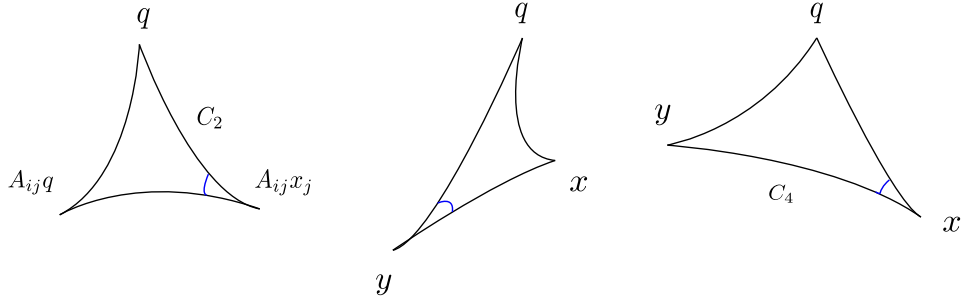}
  \caption{Triangles in Step (b), (c), and (d).}
    \label{triangles}
\end{figure}

Here is a listing of properties serving as the building blocks of the proof of Lemma \ref{keylemma}.

\begin{lemma}
\label{steps}
(a) For any constant $C>0$, there exists a constant $C_1=C_1(M,T,\psi,C)>0$ such that if
$$
\sum_{i\in V}d(x_i,q)^2\geq C_1,
$$
then
$$
\max_{ij\in E}d(x_i,A_{ij}x_j)\geq C.
$$

(b) There exists a constant $C_2=C_2(M,T,\psi)>0$ such that if
$$
d(A_{ij}x_j,q)\geq C_2,
$$
then
$$
\angle (A_{ji}q)x_j q = \angle q(A_{ij}x_j)(A_{ij}q) \leq\frac{\pi}{8}.
$$

(c) There exists a constant $C_3>0$ such that if $x,y\in\tilde M$ satisfy that
$$
d(y,q)\geq d(x,q)+C_3,
$$
then
$$\angle xyq\leq \frac{\pi}{4}.
$$

(d) There exists a constant $C_4>0$ such that if $x,y\in\tilde M$ satisfy that
$$
d(x,y)\geq C_4, \quad\text{ and }\quad d(x,q)\geq d(y,q),
$$
then
$$
\angle yxq\leq\frac{\pi}{8}.
$$

(e) For any constant $C>0$, there exists a constant $C_5=C_5(M,T,\psi,C)>0$ such that if
$$
\max_{ij\in E}d(x_i,A_{ij}x_j)\geq C_5,
$$
then there exists $ij\in E$ such that
$$
\frac{ v(x_i,q)}{| v(x_i,q)|}\cdot v(x_i,A_{ij}x_j)\geq C.
$$

(f) For any constant $C>0$, there exists a constant $C_6=C_6(M,T,\psi ,\lambda_w,C)>0$ such that if
$$
\frac{ v(x_i,q)}{| v(x_i,q)|}\cdot v(x_i,A_{ij}x_j)\geq C_6
$$
for some edge $ij\in E$, then there exists $ik\in E$ such that
$$
\frac{ v(x_i,q)}{| v(x_i,q)|}\cdot v(x_i,A_{ik}x_k)\leq -C.
$$

(g) For any constant $C>0$, there exists a constant $C_7=C_7(M,T,\psi,C)>0$ such that if
$$
\frac{ v(x_i,q)}{| v(x_i,q)|}\cdot v(x_i,A_{ik}x_k)\leq -C_7,
$$
then
$$
d(x_k,q)\geq d(x_i,q)+C.
$$

(h) For any constant $C>0$, there exists a constant $C_8=C_8(M,T,\psi,C)>0$ such that if
$$
d(x_j,q)\geq d(x_i,q)+C_8,
$$
then
$$
\frac{ v(x_j,q)}{| v(x_j,q)|}\cdot v(x_j,A_{ji}x_i)\geq C.
$$
\end{lemma}

\begin{figure}[h!]
  \includegraphics[width=0.4\linewidth]{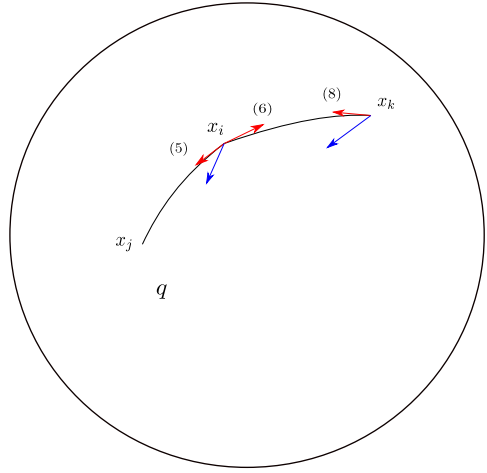}
  \caption{Vertices leaving the point $q$.}
    \label{drag}
\end{figure}

\begin{proof}[Proof of Lemma \ref{keylemma} assuming Lemma \ref{steps}]
For any $C>0$, there exists a sufficiently large constant $\tilde{C} = \tilde{C} (M,T,\psi,\lambda_w,C) $ determined from (a), (e), (f) and (g) in Lemma \ref{steps} such that if
$$
\sum_{i\in V}d(x_i,q)^2\geq \tilde{C} ,
$$
then there exist three vertices $x_i$, $x_j$, and $x_k$ shown in Figure \ref{drag} with
$$
d(x_k, q)\geq d(x_j, q) + C.
$$
Moreover, by (g), (e) and (f) of Lemma \ref{steps}, we can find another vertex $x_l$ such that
$$
d(x_l,q)\geq d(x_k, q)+C\geq d(x_j, q) + 2C,
$$
if the constant $\tilde{C} (M,T,\psi,\lambda_w,C)$ is sufficiently large.

Inductively,  we can find a sequence $i_1,...,i_{n+1}\in V$ such that
$$
d(x_{i_1},q)>d(x_{i_2},q)>...>d(x_{i_{n+1}},q).
$$
This contradicts to the fact that $V$ only has $n$ different elements.
\end{proof}

\begin{proof}[Proof of Lemma \ref{steps}]
(a)
By a standard compactness argument, the set
$$
\{\varphi\in\tilde X:\max_{ij\in E}~~length(\varphi_{ij}([0,1]))\leq C\}
$$
is a compact subset of $\tilde X$. Notice that $(x_1,...,x_n)\mapsto\varphi[x_1, ..., x_n]$ is a homeomorphism from $\tilde M^n$ to $\tilde X$ and
$$
length(\varphi_{ij}([0,1])) = d(x_i,A_{ij}x_j).
$$
Therefore,
$$
\{(x_1,...,x_n)\in\tilde M^n:\max_{ij\in E}~~d(x_i,A_{ij}x_j)\leq C\}
$$
is compact and the conclusion follows.

(b)
We claim that the constant $C_2$, which is determined by
$$
\sinh C_2=\frac{\max_{ij\in E}\sinh d(A_{ij}q,q)}{\sin\frac{\pi}{8}},
$$
is satisfactory.
Let $\triangle ABC$ be the hyperbolic triangle with the corresponding edge lengths
$$
a = d(A_{ij}x_j,q),\quad b=d(A_{ij}x_j,A_{ij}q),\quad c=d(A_{ij}q,q).
$$
Since $\tilde{M}$ is a CAT($-1$) space, it suffices to show that $\angle C\leq\pi/8$. By the hyperbolic law of sine,
$$
\sin\angle C=\frac{\sinh c\cdot\sin\angle A}{\sinh a}\leq\frac{\max_{ij\in E}\sinh d(A_{ij}q,q)\cdot 1}{\sinh C_2}=\sin\frac{\pi}{8}.
$$

(c)
We claim that the constant $C_3$ determined by
$$
\sinh C_3=\frac{1}{\sin\frac{\pi}{8}}
$$
is satisfactory.
Let $\triangle ABC$ be the hyperbolic triangle with the corresponding edge lengths
$$
a = d(x,y),\quad b=d(y,q),\quad c=d(x,q).
$$
Since $\tilde{M}$ is a CAT($-1$) space, it suffices to show that $\angle C\leq\pi/8$. By the hyperbolic law of sine
$$
\sin\angle C=\frac{\sinh c\cdot\sin\angle B}{\sinh b}\leq\frac{\sinh c}{\sinh b}\leq
\frac{\sinh c}{\sinh (c+C_3)}\leq\frac{\sinh c}{\sinh c\cdot\sinh C_3}=\sin\frac{\pi}{8}.
$$

(d)
We claim that the constant $C_4$ determined by
$$
\sin^2\frac{\pi}{8}\cdot \cosh C_4=2
$$
is satisfactory.
Let $\triangle ABC$ be the hyperbolic triangle with the corresponding edge lengths
$$
a = d(x,y),\quad b=d(y,q),\quad c=d(x,q).
$$
Since $\tilde{M}$ is a CAT($-1$) space, it suffices to show that $\angle B\leq\pi/8$. By the hyperbolic law of cosine,
$$
\cos A = -\cos B\cos C+\sin B\sin C\cosh a.
$$
Then,
$$
2\geq\sin B\sin C\cosh a\geq\sin B\sin C\cosh C_4=2\cdot\frac{\sin B\sin C}{\sin^2\frac{\pi}{8}}\geq2\cdot
\frac{\sin^2 B}{\sin^2\frac{\pi}{8}}.
$$
Thus, $\angle B\leq\pi/8$.

\begin{figure}[h!]
  \includegraphics[width=0.65\linewidth]{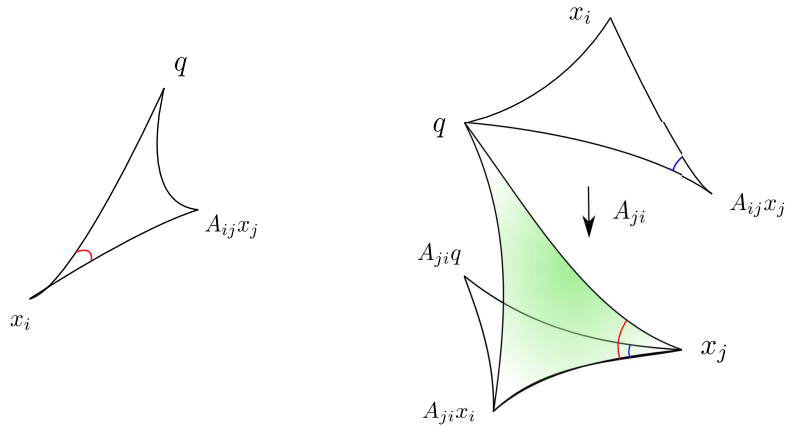}
  \caption{Triangles in Step (5).}
    \label{step5}
\end{figure}

(e)
We claim that the constant $C_5$ determined by
$$
C_5=\max\{C_4,2C_2,\sqrt 2C\}
$$
is satisfactory. Assume $ij\in E$ and $d(x_i,A_{ij}x_j)\geq C_5$, and we have two cases shown in Figure \ref{step5}.

If $d(x_i,q)\geq d(A_{ij}x_j,q)$, then by part (d)
$$
\angle (A_{ij}x_j)x_iq\leq\frac{\pi}{8}\leq\frac{\pi}{4},
$$
and
$$
\frac{ v(x_i,q)}{| v(x_i,q)|}\cdot v(x_i,A_{ij}x_j)=\cos (\angle (A_{ij}x_j)x_iq) \cdot d(x_i,A_{ij}x_j)\geq \frac{1}{\sqrt 2}C_5\geq C.
$$

If $d(x_i,q)\leq d(A_{ij}x_j,q)$,  then $d(A_{ij}x_j,q)\geq C_2$. By part (b) and part (d),
$$
\angle(A_{ji}q)x_j q \leq\frac{\pi}{8},\quad \text{ and }\quad
\angle (A_{ji}q)x_j(A_{ji}x_i)=\angle q(A_{ij}x_j)x_i\leq\frac{\pi}{8},
$$
and
$\angle qx_j(A_{ji}x_i)\leq\pi/4$. Therefore,
$$
\frac{ v(x_j,q)}{| v(x_j,q)|}\cdot v(x_j,A_{ji}x_i)=\cos (\angle (A_{ji}x_j)x_jq) \cdot d(x_j,A_{ji}x_i)\geq \frac{1}{\sqrt 2}C_5\geq C.
$$

(f)
We claim that the constant $C_6$ determined by
$$
C_6= n\lambda_w\cdot C
$$
is satisfactory.
If not, for any $ik\in E$, we have
$$
\frac{ v(x_i,q)}{| v(x_i,q)|}\cdot v(x_i,A_{ik}x_k)> -C.
$$
Then
$$
0\geq\frac{ v (x_i,q)}{| v (x_i,q)|}\cdot\sum_{ik\in E}w_{ik}  v(x_i,A_{ik}x_k)
> w_{ij}C_6+\sum_{ik\in E}w_{ik}(-C)
$$
$$
\geq w_{ij}C_6+\sum_{ik\in E}\lambda_w w_{ij}(-C)\geq
w_{ij}(C_6-n\lambda_w C)\geq0,
$$
and it is a contradiction.

(g) We claim that $C_7=C + \max_{ij\in E}d(A_{ij}q,q)$ is satisfactory. Notice that
$$d(A_{ij}x_j, q) = d(x_j, A_{ji}q) \leq d(x_j, q) + d(q, A_{ji}q) \leq d(x_j, q) +\max_{ij\in E}d(A_{ij}q,q).$$
By part (1) of Lemma \ref{CAT},
$$
d(A_{ij}x_j,q)
\geq| v(x_i,A_{ij}x_j)- v(x_i,q)|
\geq -\left( v(x_i,A_{ij}x_j)- v(x_i,q)\right)\cdot \frac{ v(x_i,q)}{| v(x_i,q)|}
$$
$$
=C_7+| v(x_i,q)|=C_7+d(x_i,q).
$$
Then
$$d(x_j, q) - d(x_i, q) \geq C_7 - \max_{ij\in E}d(A_{ij}q,q) = C.$$

(h)
We claim that the constant $C_8$ determined by
$$
C_8=\max\{C_3,\sqrt 2 C\}+\max_{ij\in E}d(A_{ij}q,q)
$$
is satisfactory. Notice that
$$
d(x_j,q)\geq d(x_i,q)+C_8\geq d(x_i,A_{ij}q)-d(A_{ij}q,q)+C_8\geq d(A_{ji}x_i,q)+\max\{C_3,\sqrt2C\}.
$$
Then by part (c), $\angle (A_{ji}x_i)x_jq\leq\pi/4$, and by the triangle inequality,
$$d(x_j, A_{ji}x_i) \geq d(x_j, q) - d(A_{ji}x_i, q) \geq \sqrt{2}C. $$
Therefore,
$$
\frac{ v(x_j,q)}{| v(x_j,q)|}\cdot v(x_j,A_{ji}x_i)=\cos (\angle (A_{ji}x_i)x_jq) \cdot d(x_j,A_{ji}x_i)\geq \frac{1}{\sqrt 2}\cdot \sqrt2C=C.
$$
\end{proof}

\subsection{Proof of the Continuity}
\begin{proof}[Proof of the continuity part of Theorem \ref{existence}]
If not, there exists $\epsilon>0$ and a weight $w$ and a sequence of weights $w^{(k)}$ such that

\begin{enumerate}
	\item $w^{(k)}$ converge to $w$, and
	\item $d_{\tilde X}(\Phi(w^{(k)}),\Phi(w))\geq\epsilon$ for any $k\geq1$.
\end{enumerate}

By the stronger existence result Theorem \ref{proper}, the sequence $\Phi(w^{(k)})$ are in some fixed compact subset $K'$ of $\tilde X$. By picking a subsequence, we may assume that $\Phi(w^{(k)})$ converge to some $\varphi\in \tilde X$. Since $\Phi(w^{(k)})$ is $w^{(k)}$-balanced,
then by the continuity of the residue vectors $r_i$,
$\varphi$ is $w$-balanced, and thus $\Phi(w)=\varphi$, which is contradictory to that $\Phi(w^{(k)})$ does not converge to $\Phi(w)$.
\end{proof}

\section{Proof of Theorem \ref{embedding}}
\subsection{Set up and preparations}
Assume $\varphi\in\tilde X$ is $w$-balanced for some weight $w$, and we will prove that $\varphi$ is an embedding. Recall that $q_i=\varphi(i)$ for each $i\in V$, and denote $l_{ij}$ as the length of $\varphi_{ij}([0,1])$ for any $ij\in E$. It is not difficult to show that $\varphi$ has a  continuous extension $\tilde\varphi$ defined on $|T|$, such that for any triangle $\sigma\in F$ a continuous lifting map $\Phi_\sigma$ of  $\tilde\varphi|_\sigma$ from $\sigma$ to $\tilde M$ will

\begin{enumerate}
	\item map $\sigma$ to a geodesic triangle in $\tilde M$ homeomorphically if $\varphi(\partial\sigma)$ does not degenerate to a geodesic, and
	\item map $\sigma$ to $\Phi_\sigma(\partial\sigma)$ if $\varphi(\partial\sigma)$ degenerates to a geodesic.
\end{enumerate}

The main tool to prove Theorem \ref{embedding} is the Gauss-Bonnet formula. We will need to define the inner angles for each triangle in $\varphi(T^{(1)})$, even for the degenerate triangles. A convenient way is to assign a ``direction" to each edge, even for the degenerate edges with zero length.

\begin{definition}
 A \textit{direction field} is a map $v:\vec E\rightarrow TM$ satisfying that
\begin{enumerate}
	\item $ v_{ij}\in T_{q_i}M$ for any $(i,j)\in \vec E$, and
	\item $| v_{ij}|=1$ for any $(i,j)\in \vec E$.
\end{enumerate}
Given a direction field $v$, define the inner angle of the triangle $\sigma = \triangle ijk$ at the vertex $i$ as
$$
\theta^i_\sigma=\theta^i_\sigma(v)=\angle  v_{ij}0  v_{ik}=\arccos (v_{ij} \cdot v_{ik}),
$$
where $0$ is the origin and $\angle  v_{ij}0  v_{ik}$ is the angle between $ v_{ij}$ and $ v_{ik}$ in $T_{q_i}M$.
\end{definition}

A direction field $v$ assigns a unit tangent vector in $T_{q_i}M$ to each directed edge starting from $i$, and determines the inner angles in $T$.

\begin{definition}
A direction field $v$ is \emph{admissible} if
\begin{enumerate}
		\item
	$$ v_{ij} = \frac{\varphi_{ij}'(0)}{l_{ij}}$$
	if $l_{ij}>0$, and
	\item $ v_{ij}=- v_{ji}$ in $T_{q_i}M =T_{q_j}M $ if $l_{ij}=0$, and
	\item for a fixed vertex $i\in V$, if $l_{ij}=0$ for any neighbor $j$ of $i$, then there exist neighbors $j$ and $k$ of $i$ such that $ v_{ij}=- v_{ik}$, and
	\item if $\sigma=\triangle ijk\in F$ and $l_{ij}=l_{jk}=l_{ik}=0$, then $\theta^i_\sigma(v)+\theta^j_\sigma(v)+\theta^k_\sigma(v)=\pi$.
\end{enumerate}
\end{definition}

An admissible direction field encodes the directions of the non-degenerate edges in $\varphi(T^{(1)})$, and the induced angle sum of a degenerate triangle is always $\pi$. Then for any admissible $v$ and triangle $\sigma\in F$, by the Gauss-Bonnet formula
\begin{equation}
\label{GB}
\pi=\sum_{i\in \sigma}\theta^i_\sigma(v)-\int_{\Phi_\sigma(\sigma)}KdA
\geq\sum_{i\in \sigma}\theta^i_\sigma(v)-\int_{\tilde\varphi(\sigma)}KdA.
\end{equation}
Here $dA$ is the area form on $(\tilde M,\tilde g)$ or $(M,g)$.

The concept of the direction field is similar to the \textit{discrete one form} defined in \cite{gortler2006discrete}. 

\subsection{Proof of Theorem \ref{embedding}}
The proof of Theorem \ref{embedding} uses the four lemmas below. We will postpone their proofs to the the subsequent subsections.
\begin{lemma}
\label{star}
If $v$ is admissible and $\theta=\theta(v)$, then for any $i\in V$,
$$
\sum_{\sigma:i\in \sigma}\theta^i_\sigma=2\pi,
$$
and for any $\sigma,\sigma'\in F$,
$\tilde\varphi(\sigma)\cap \tilde\varphi(\sigma')$ has area $0$.
\end{lemma}

Based on Lemma \ref{star}, if admissible direction fields exist, the image of the star of each vertex determined by $\tilde\varphi$ does not contain any flipped triangles overlapping with each other. If $\tilde\varphi(\sigma)$ does not degenerate to a geodesic arc for any triangle $\sigma\in F$, then $\tilde\varphi$ is locally homeomorphic and thus globally homeomorphic as a degree-one map. Therefore, we only need to exclude the existence of degenerate triangles.

Define an equivalence relation on $V$ as follows. Two vertices $i,j$ are equivalent if there exists a sequence of vertices  $i=i_0,i_1,...,i_k=j$ such that $l_{i_0i_1}=...=l_{i_{k-1}i_k}=0$. This equivalence relation introduces a partition $V=V_1\cup...\cup V_m.$ Denote $y_k\in M$ as the only point in $\varphi(V_k)$. For any $x\in M$ and $u,v\in T_xM$, denote $u\| v$ as $u$ and $v$ are parallel, i.e., there exists $(\alpha,\beta)\neq(0,0)$ such that $\alpha u+\beta v=0$.

The following Lemma \ref{prescribe} shows that there are plenty of choices of admissible direction fields.

\begin{lemma}
\label{prescribe}
For any $ v_1\in T_{y_1}M,..., v_m\in T_{y_m}M$, there exists an admissible $v$ such that $ v_{ij}\| v_k$ if $i\in V_k$ and $l_{ij}=0$.
\end{lemma}

The following Lemma \ref{flatstar} shows that for any $V_k$ with at least two vertices, the image of its ``neighborhood" lies in a geodesic.

\begin{lemma}
\label{flatstar}
If $|V_k|\geq2$, then there exists $ v_k\in T_{y_k}M$ such that $ v_k\| \varphi_{ij}'(0)$ if $i\in V_k$ and $l_{ij}>0$.
\end{lemma}

Now let $v_k$ be as Lemma \ref{flatstar} if $|V_k|\geq2$, and arbitrary if $|V_k|=1$. Then construct an admissible direction field $v$ as in Lemma \ref{prescribe}, with induced inner angles $\theta^i_\sigma=\theta^i_\sigma(v)$. If the image of a triangle $\sigma$ under $\varphi$ degenerates to a geodesic, then its inner angles $\theta^i_\sigma$ are $\pi$ or $0$.
Let $F' \neq \emptyset$ be the set of degenerate triangles under $\varphi$.

\begin{lemma}
 \label{degenerate}
 If $\sigma\in F'$, $i\in\sigma$, and $\theta^i_\sigma=\pi$, then $\sigma'\in F'$ for any $\sigma'$ in the star neighborhood of the vertex $i$.
 \end{lemma}

Let $\Omega$ be a connected component of the interior of $\cup \{\sigma:\sigma\in F'\}\subset|T|$, and $\tilde\Omega$ be the completion of $\Omega$ under the natural path metric on $\Omega$.
Notice that $\tilde \Omega$ could be different from the closure of $\Omega$ in $M$.

Since $\tilde\varphi$ is surjective, $F'\neq F$ and $\Omega\neq|T|$ and $\tilde\Omega$ has non-empty boundary.
Then $\tilde\Omega$ is a connected surface with a natural triangulation $T'=(V',E',F')$, and
$$
\chi(\tilde\Omega)=2-2\times(\text{genus of } \tilde\Omega)-\#\{\text{boundary components of $\tilde\Omega$}\}\leq1.
$$

Assume $V'_I$ is the set of interior vertices, and $V'_B$ is the set of boundary vertices, and $E'_I$ is the set of interior edges, and $E'_B$ is the set of boundary edges of $\tilde{\Omega}$. Then $|V'_B|=|E'_B|$, and by Lemma \ref{degenerate}, if $i\in V'_B$ and $i\in \sigma$, then $\theta^i_\sigma=0$. Therefore,
$$
\pi|F'|=\sum_{\sigma\in F',i\in \sigma}\theta^i_{\sigma}=\sum_{i\in V'_I}\sum_{\sigma\in F':i\in \sigma}\theta^i_\sigma=2\pi|V_I'|.
$$
Thus,
\begin{align*}
	1\geq \chi(\tilde{\Omega})= & |V'|-|E'|+|F'|=|V'_I|+|V'_B|-|E'_I|-|E'_B|+|F'| \\
	= & |V'_B|-|E'_I|-|E'_B|+\frac{3}{2}|F'|=-|E'_I|+\frac{3}{2}|F'| \\
	= & -|E'_I|+\frac{1}{2}(|E'_B|+2|E'_I|) =\frac{1}{2}|E_B'|.
\end{align*}
Therefore, $|V'_B|=|E'_B|\leq 2$. Since $\tilde\Omega$ has non-empty boundary, $|E_B'|=1$ or $2$. In either case, it contradicts with the fact that $T$ is a simplicial complex.

\subsection{Proof of Lemma \ref{star}}

We claim that for any $i\in V$,
$$
\sum_{\sigma:i\in \sigma}\theta^i_\sigma\geq2\pi.
$$
If $l_{ij}=0$ for any neighbor $j$ of $i$, this is a consequence of condition (3) in the definition of an admissible direction field. If $l_{ij}\neq0$, by the $w$-balanced condition,
$\{\varphi_{ij}'(0)/l_{ij}:ij\in E \}$ should not be contained in any open half unit circle, and the angle sum around $i$ should be at least $2\pi$.

By  the fact that $\tilde\varphi$ is surjective and equation (\ref{GB}), we have
$$
\sum_{i\in V}(2\pi-\sum_{\sigma:i\in \sigma}\theta^i_\sigma)+\sum_{\sigma\in F}\int_{\tilde\varphi (\sigma)}KdA
\leq\sum_{\sigma\in F}\int_{\tilde\varphi (\sigma)}KdA\leq \int_M KdA=2\pi\chi(M),
$$
and
$$
\sum_{i\in V}(2\pi-\sum_{\sigma:i\in \sigma}\theta^i_\sigma)+\sum_{\sigma\in F}\int_{\tilde\varphi (\sigma)}KdA
\geq2\pi|V|-\sum_{\sigma\in F}(\sum_{i\in \sigma}\theta^i_\sigma-\int_{\tilde{\varphi}(\sigma)}KdA)=2\pi\chi(M).
$$
Hence, the inequalities above are equalities. This fact implies that
$$\sum_{i\in V}(2\pi-\sum_{\sigma:i\in \sigma}\theta^i_\sigma) = 0.$$
Since each term in this summation is non-positive, then $\sum_{\sigma:i\in \sigma}\theta^i_\sigma = 2\pi$. The statement on the area follows similarly.

\subsection{Proof of Lemma \ref{prescribe}}
We claim that for any $k$, there exists a map $h:V_k\rightarrow\mathbb R$ such that
\begin{enumerate}
	\item $h(i)\neq h(j)$ if $i\neq j$, and
	\item for a fixed $i\in V_k$, if $l_{ij}=0$ for any neighbor $j$ of $i$, then there exist neighbors $j$ and $j'$ of $i$ in $V_k$ such that $h(j)<h(i)<h(j')$.
\end{enumerate}
Given such $h$, set $v$ as

$$ v_{ij} = \begin{cases}
      \text{sgn}[h(j)-h(i)]\cdot  v_k, & \text{ if } i\in V_k \text{ and } l_{ij}=0,\\
      \varphi_{ij}'(0), & \text{ if } l_{ij}>0,
   \end{cases}
$$
where $\text{sgn}$ is the sign function. It is easy to verify that such $v$ is satisfactory.

To construct such function $h$, we will prove the following lemma, which is more general than our claim.
\begin{lemma}
Assume $G=(V',E')$ is a subgraph of the $1$-skeleton $T^{(1)}$, and $E'\neq E$. Denote
$$
int(G)=\{i\in V':ij\in E \Rightarrow ij\in E'\},
$$
and $\partial G=V'-int(G)$. Then there exists $h:V'\rightarrow \mathbb R$ such that
\begin{enumerate}
	\item  $h(i)\neq h(j)$ if $i\neq j$, and
	\item for any $i\in int(G)$ there exist neighbors $j$ and $j'$ of $i$ in $V'$ such that $h(j)<h(i)<h(j')$.
\end{enumerate}
\end{lemma}

\begin{proof}
We prove by the induction on the size of $V'$. The case $|V'|=1$ is trivial. For the case $|V'|\geq2$, first notice that $|\partial G| \geq 2$ for any proper subgraph $G$ of $T^{(1)}$. Assign distinct values $\bar h(i)$ to each $i\in \partial G$, then solve the discrete harmonic equation
$$
\sum_{j:ij\in E} (\bar h(j)-\bar h(i))=0, \quad\forall i\in int(G),
$$
with the given Dirichlet boundary condition on $\partial G$.

Let $s_1<...<s_k$ be the all distinct values that appear in $\{\bar h(i):i\in V'\}$. Then consider the subgraphs $G_i = (V'_i, E'_i)$ defined as
$$
V_i'=\{j\in V':h(j)=s_i\},
$$
and
$$
E_i'=\{jj'\in E': j,j'\in V_i'\}.
$$
Notice that $|\partial G|\geq2$, so $k\geq2$ and $|V_i'|<|V'|$ for any $i=1,...,k$.
By the induction hypothesis, there exists a function $h_i:V_i'\rightarrow \mathbb R$ such that
\begin{enumerate}
	\item $h_i(j)\neq h_i(j')$ if $j\neq j'$, and
	\item for any $j\in int(G_i)$ there exist neighbors $j',j''$ of $j$ in $V'_i$ such that $h_i(j')<h_i(j)<h_i(j'')$.
\end{enumerate}
Define $h_i(j)=0$ if $j\notin V_i'$, then for sufficiently small positive $\epsilon_1,...,\epsilon_k$,
$${h} = \bar h+\sum_{i=1}^k \epsilon_ih_i$$
is a desirable function.
\end{proof}

\subsection{Proof of Lemma \ref{flatstar}}
It amounts to prove that if $i,i'\in V_k$ and $ij,i'j'\in E$ and $l_{ij}>0$ and $l_{i'j'}>0$, then $\varphi'_{ij}(0)\|\varphi'_{i'j'}(0)$.
Let
$$
D=\left(\cup_{i,i',i''\in V_k}\triangle i i' i''\right) \cup \left(\cup_{i,i'\in V_k}ii'\right),
$$
which is a closed path-connected set in $|T|$.
For any $i\in V_k$, $i\in \partial D$ if and only if there exists $ij\in E$ with $l_{ij}>0$. Therefore, it suffices to prove that

\begin{enumerate}
	\item for any $i\in V_k$ and edges $ij,ij'$ with $l_{ij}>0$ and $l_{ij'}>0$, $\varphi_{ij}'(0)\|\varphi_{ij'}'(0)$, and
	\item for any $ij\in E$ satisfying that $ij\subset\partial D$, there exists $m\in V-V_k$ such that $\triangle ijm\in F$, and thus $\varphi_{im}'(0)=\varphi_{jm}'(0)$, and
	\item $\partial D$ is connected.
\end{enumerate}

For part (1), if it is not true, then there exists $i\in V_k$ and $ij\in E$ and $ij'\in E$ such that $l_{ij}>0$ and $l_{ij'}>0$ and $ \varphi_{ij}'(0)$ is not parallel to $\varphi_{ij'}'(0)$. Assuming this claim,  by the $w$-balanced condition, there exists $ij''\in E$ with $l_{ij''}>0$, and the three vectors $ \varphi_{ij}'(0),  \varphi_{ij'}'(0),  \varphi_{ij''}'(0)$ are not contained in any closed half-space in $T_{q_i}M$. Assume $im\in E$ and $l_{im}=0$, and without loss of generality,  $ij,im,ij',ij''$ are ordered counter-clockwise in the one-ring neighborhood of $i$ in $T$. By Lemma \ref{prescribe}, there exists an admissible $v$ such that $ v_{im}\| \varphi_{ij''}'(0)$.  By possibly changing a sign, we may assume that $ v_{im}=\varphi_{ij''}'(0)/l_{ij''}$.

\begin{figure}[h!]
  \includegraphics[width=0.6\linewidth]{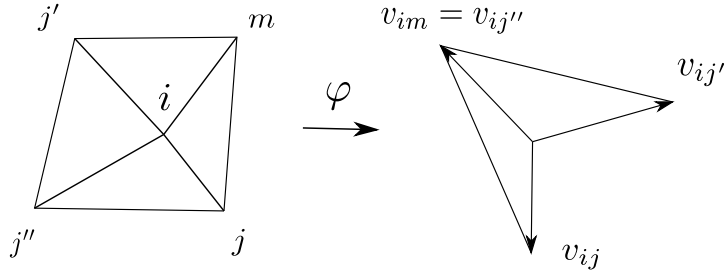}
  \caption{Overlapping triangles lead to angle surplus.}
    \label{balanced}
\end{figure}
Then as Figure \ref{balanced} shows, a contradiction follows
$$
2\pi=\sum_{\sigma i\in \sigma}\theta^i_\sigma\geq\angle  v_{ij}0 v_{im}+\angle  v_{im}0 v_{ij'}+\angle  v_{ij'}0  v_{ij''}+\angle  v_{ij''}0 v_{ij}
$$
$$
=2\angle  v_{ij}0 v_{ij''}+2\angle  v_{ij''}0 v_{ij'}>2\pi.
$$

Part (2) is straightforward and we will prove part (3).
By our assumption on the extension $\tilde\varphi$, $\tilde\varphi(D)$ contains only one point, then the embedding map $i_D=\psi^{-1}\circ(\psi|_D)$ from $D$ to $|T|$ is homotopic to the constant map $\psi^{-1}\circ(\tilde\varphi|_D)$, meaning that $D$ is contractible in $|T|$. If $\partial D$ contains at least two boundary components, then it is not difficult to show that $|T|-D$ has a connected component $D'$ homeomorphic to an open disk. Let $\Phi_{ D}:D\to\tilde M$ be a lifting of $\tilde\varphi|_{D}$. 
Then $\Phi_{ D}(\partial D')\subset\Phi_{D}(D)$ contains only a single point $x\in\tilde M$. Then by the $w$-balanced condition, it is not difficult to derive a maximum principle and show that $\tilde\varphi|_{ D'}$ equals to constant $x$. Then by the definition of $D$ it is easy to see that $D'$ should be a subset of $D$. It is contradictory.

\subsection{Proof of Lemma \ref{degenerate}}
Assume $ij$ and $ij'$ are two edges in $\sigma$. If the conclusion is not true, then there exists $ik\in E$ such that $l_{ik}>0$ and $ v_{ik}$ is not parallel to $ v_{ij}$. Notice that $ v_{ij}=- v_{ij'}$, and we have
$$
2\pi=\sum_{\sigma\in F:i\in \sigma}\theta^i_\sigma\geq \angle  v_{ij}0  v_{ij'}+\angle  v_{ij}0 v_{ik}+\angle   v_{ij'}0 v_{ik}=2\pi.
$$
Then the equality holds in the above inequality, and for any $ik'\in E$, $ v_{ik'}$ should be on the half circle that contains $ v_{ij}, v_{ik}, v_{ij'}$. Let $ v_m$ be the middle point of this half circle, then
$$
 v_m\cdot\sum_{j:ij\in E}w_{ij}l_{ij} v_{ij}\geq w_{ik}l_{ik} v_m\cdot  v_{ik}>0.
$$
This contradicts to the fact that $\varphi$ is $w$-balanced.

\bibliography{ref}

\providecommand{\bysame}{\leavevmode\hbox to3em{\hrulefill}\thinspace}
\providecommand{\MR}{\relax\ifhmode\unskip\space\fi MR }
\providecommand{\MRhref}[2]{%
  \href{http://www.ams.org/mathscinet-getitem?mr=#1}{#2}
}
\providecommand{\href}[2]{#2}
\begin{thebibliography}{10}

\bibitem{awartani1987spaces}
Marwan Awartani and David~W Henderson, \emph{Spaces of geodesic triangulations
  of the sphere}, Transactions of the American Mathematical Society
  \textbf{304} (1987), no.~2, 721--732.

\bibitem{bloch1984space}
Ethan~D Bloch, Robert Connelly, and David~W Henderson, \emph{The space of
  simplexwise linear homeomorphisms of a convex 2-disk}, Topology \textbf{23}
  (1984), no.~2, 161--175.

\bibitem{bridson2013metric}
Martin~R Bridson and Andr{\'e} Haefliger, \emph{Metric spaces of non-positive
  curvature}, vol. 319, Springer Science \& Business Media, 2013.

\bibitem{burago2001course}
Dmitri Burago, Iu~D Burago, Yuri Burago, Sergei Ivanov, Sergei~V Ivanov, and
  Sergei~A Ivanov, \emph{A course in metric geometry}, vol.~33, American
  Mathematical Soc., 2001.

\bibitem{cairns1944isotopic}
Stewart~S Cairns, \emph{Isotopic deformations of geodesic complexes on the
  2-sphere and on the plane}, Annals of Mathematics (1944), 207--217.

\bibitem{chambers2021morph}
Erin~Wolf Chambers, Jeff Erickson, Patrick Lin, and Salman Parsa, \emph{How to
  morph graphs on the torus}, Proceedings of the 2021 ACM-SIAM Symposium on
  Discrete Algorithms (SODA), SIAM, 2021, pp.~2759--2778.

\bibitem{connelly1983problems}
Robert Connelly, David~W Henderson, Chung~Wu Ho, and Michael Starbird, \emph{On
  the problems related to linear homeomorphisms, embeddings, and isotopies},
  Continua, decompositions, manifolds, 1983, pp.~229--239.

\bibitem{de1991comment}
Y~Colin de~Verdiere, \emph{Comment rendre g{\'e}od{\'e}sique une triangulation
  d'une surface}, L'Enseignement Math{\'e}matique \textbf{37} (1991), 201--212.

\bibitem{earle1969fibre}
Clifford~J Earle, James Eells, et~al., \emph{A fibre bundle description of
  teichm{\"u}ller theory}, Journal of Differential Geometry \textbf{3} (1969),
  no.~1-2, 19--43.

\bibitem{floater2003one}
Michael Floater, \emph{One-to-one piecewise linear mappings over
  triangulations}, Mathematics of Computation \textbf{72} (2003), no.~242,
  685--696.

\bibitem{floater2003mean}
Michael~S Floater, \emph{Mean value coordinates}, Computer Aided Geometric
  Design \textbf{20} (2003), no.~1, 19--27.

\bibitem{gaster2018computing}
Jonah Gaster, Brice Loustau, and L{\'e}onard Monsaingeon, \emph{Computing
  discrete equivariant harmonic maps}, arXiv preprint arXiv:1810.11932 (2018).

\bibitem{gortler2006discrete}
Steven Gortler, Craig Gotsman, and Dylan Thurston, \emph{Discrete one-forms on
  meshes and applications to 3d mesh parameterization}, Computer Aided
  Geometric Design \textbf{23} (2006), no.~2, 83--112.

\bibitem{hass2012simplicial}
Joel Hass and Peter Scott, \emph{Simplicial energy and simplicial harmonic
  maps}, Asian Journal of Mathematics \textbf{19} (2015), no.~4, 593--636.

\bibitem{1910.03070}
Yanwen Luo, \emph{Spaces of geodesic triangulations of surfaces}, arXiv
  preprint arXiv:1910.03070. \textit{Submitted} (2019).

\bibitem{smale1959diffeomorphisms}
Stephen Smale, \emph{Diffeomorphisms of the 2-sphere}, Proceedings of the
  American Mathematical Society \textbf{10} (1959), no.~4, 621--626.

\bibitem{tutte1963draw}
William~Thomas Tutte, \emph{How to draw a graph}, Proceedings of the London
  Mathematical Society \textbf{3} (1963), no.~1, 743--767.

\end{thebibliography}
\bibliographystyle{amsplain}

\end{document}